\documentclass[reqno,12pt]{amsart}
\usepackage{color,enumerate}
\usepackage{amssymb,amsmath}
\usepackage[titletoc]{appendix}
\newtheorem{theorem}{Theorem}[section]
\newtheorem{lemma}[theorem]{Lemma}

\newtheorem{proposition}[theorem]{Proposition}

\theoremstyle{remark}

\newtheorem{remark}[theorem]{Remark}
\newtheorem{definition}[theorem]{Definition}

\numberwithin{equation}{section}
\def\le{\,{\leqslant}\,}
\def\ge{\,{\geqslant}\,}

\def\la{{\lambda}}
\def\N{{\mathbb{N}}}
\def\Z{{\mathbb{Z}}}
\def\R{{\mathbb{R}}}
\def\C{{\mathbb{C}}}
\def\T{{\mathbb{T}}}
\def\Bc{{\mathcal{B}}}
\def\Hc{{\mathcal{H}}}
\def\Sc{{\mathcal{S}}}
\def\Xc{{\mathcal{X}}}
\def\Ac{{\mathcal{A}}}
\def\T{{\mathcal{T}}}

\def\tr{{\mathrm{Tr}\,}}

\def\sym{{\,\mathrm{Sym}}}

\newcommand{\norm}[1]{\Vert#1\Vert}
\newcommand{\bignorm}[1]{\bigl\Vert#1\bigr\Vert}
\newcommand{\Bignorm}[1]{\Bigl\Vert#1\Bigr\Vert}

\begin{document}

\address{C.L.: Laboratoire de Mathématiques de Besan\c con, UMR 6623, CNRS, Universit\'e Bourgogne Franche-Comt\'e,
25030 Besan\c{c}on Cedex, FRANCE}\email{clemerdy@univ-fcomte.fr}
\address{A.S.: Department of Mathematics and Statistics, MSC01 1115, University of New Mexico, Albuquerque, NM 87131, USA}
\email{skripka@math.unm.edu}

\date{\today}

\title[Higher order differentiability in Schatten norms]{Higher order differentiability
of operator functions in Schatten norms}

\author[C. Le Merdy]{Christian Le Merdy$^{*}$}

\author[A. Skripka]{Anna Skripka$^{**}$}

\thanks{\footnotesize $^{*}$Research supported by the French ``Investissements d'Avenir" program,
project ISITE-BFC (contract ANR-15-IDEX-03)}
\thanks{\footnotesize $^{**}$Research supported in part by NSF grants DMS-1500704 and DMS-1554456}

\subjclass[2010]{47B49, 47B10, 
46L52}

\keywords{Differentiation of operator functions, Schatten-von Neumann classes}

\maketitle

\begin{abstract}
We establish the following results on higher order $\Sc^p$-differentiability, $1<p<\infty$, of the operator
function arising from a continuous scalar function $f$ and self-adjoint operators defined on a fixed separable Hilbert space:
\begin{enumerate}[(i)]
\item $f$ is $n$ times continuously Fr\'{e}chet $\Sc^p$-differentiable at every bounded self-adjoint operator if and only if $f\in C^n(\R)$;
\item if $f',\ldots,f^{(n-1)}\in C_b(\R)$ and $f^{(n)}\in C_0(\R)$, then $f$ is $n$ times continuously
Fr\'{e}chet $\Sc^p$-differentiable at every self-adjoint operator;
\item if $f',\ldots,f^{(n)}\in C_b(\R)$, then $f$ is $n-1$ times continuously Fr\'{e}chet $\Sc^p$-differen-tiable and $n$ times G\^{a}teaux $\Sc^p$-differentiable at every self-adjoint operator;
\end{enumerate}
We also prove that if $f\in B_{\infty1}^n(\R)\cap B_{\infty1}^1(\R)$, then $f$ is $n$ times continuously
Fr\'{e}chet $\Sc^q$-differentiable, $1\le q<\infty$, at every self-adjoint operator.
These results generalize and extend analogous results of \cite{KPSS} to arbitrary $n$ and unbounded operators as well as substantially extend the results of \cite{ACDS,CLMSS,Peller2006} on higher order $\Sc^p$-differentiability of $f$ in a certain Wiener class, G\^{a}teaux $\Sc^2$-differentiability of $f\in C^n(\R)$ with $f',\ldots,f^{(n)}\in C_b(\R)$, and G\^{a}teaux $\Sc^q$-differentiability of $f$ in the intersection of the Besov classes $B_{\infty1}^n(\R)\cap B_{\infty1}^1(\R)$. As an application, we extend $\Sc^p$-estimates for operator Taylor remainders to a broad set of symbols. Finally, we establish explicit formulas for Fr\'{e}chet differentials and G\^{a}teaux derivatives.
\end{abstract}

\section{Introduction}
\label{s1}

Differentiability is one of natural properties studied in theory of functions, in particular, operator functions. Pioneering results on differentiability of operator functions were obtained in \cite{DK1956} under restrictive assumptions on functions and operators. These results were substantially refined and extended in \cite{BS3,S1977,Peller1985,Peller1990,ABF,KS96,Ped,KS,DePS,Peller2006,ACDS,KPSS,PSS} in response to development of perturbation theory, but some problems remained open. In this paper we answer several questions on higher order differentiability of operator functions in Schatten $\Sc^p$-norms, $1<p<\infty$. The exponent $p$ in this paper is reserved for a number in the interval $(1,\infty)$.

Initially derivatives of operator functions were studied in the operator norm, but the same results hold in the $\Sc^q$-norm, $1\le q<\infty$. It was proved in \cite[Theorem 2]{Peller1990} that if $f$ is an element of the Besov space $B_{\infty1}^1(\R)$, then $f$ is G\^{a}teaux differentiable with respect to the $\Sc^q$-norm and the operator norm at every self-adjoint operator, yet a stronger sufficient condition was derived in \cite{ABF} in terms of complex analysis. It follows from \cite{BS3} and properties of Besov spaces, that $f\in B_{\infty1}^1(\R)$ is Fr\'{e}chet $\Sc^q$-differentiable, $1\le q<\infty$, at every self-adjoint operator (see Theorem \ref{Th33Besov}) and Fr\'{e}chet differentiable with respect to the operator norm at every bounded self-adjoint operator. The necessary condition $f\in B_{11}^1(\R)$ for differentiability of $f$ in the $\Sc^1$-norm was obtained in \cite[Theorem 3]{Peller1990} via nuclearity criterion for Hankel operators. In particular, the condition $f\in C^1(\R)$ is not sufficient for G\^{a}teaux $\Sc^1$-differentiability of $f$ \cite[Theorem 8]{Peller1985} (see also \cite{YuBF}). It was established in
\cite[Theorem 5.6]{Peller2006} that if $f$ is in the intersection of Besov classes $B_{\infty1}^n(\R)\cap B_{\infty1}^1(\R)$,
then $f$ is $n$ times G\^{a}teaux differentiable with respect to the $\Sc^q$-norm and the operator norm. Notes on higher order Fr\'{e}chet $\Sc^1$-differentiability follow below.

The set of Fr\'{e}chet differentiable functions with respect to the $\Sc^p$-norm is larger than the set of $\Sc^1$-differentiable functions.
It was established in \cite[Theorem 7.17]{KPSS} that $f$ is Fr\'{e}chet $\Sc^p$-differentiable at every bounded self-adjoint operator if and only if $f\in C^1(\R)$, while it was demonstrated in \cite[Example 7.20]{KPSS} that the assumption ``$f\in C^1(\R)$ and $f'$ is bounded" is not sufficient for Fr\'{e}chet differentiability of $f$ at an arbitrary unbounded operator.
It was proved in \cite[Theorem 5.5]{ACDS} that a function $f$ is $n$ times Fr\'{e}chet $\Sc^p$-differentiable (in fact, differentiable with respect to any symmetric operator ideal norm  with so called property (F)) at every self-adjoint operator if $f$ is in the Wiener class $W_{n+1}(\R)$, which consists of $f\in C^n(\R)$ with $f^{(j)}$ the Fourier transform of a finite measure on $\R$, $j=0,\ldots,n+1$. It was proved in \cite[Theorem 4.1]{CLMSS} that a function $f$ in $C^n(\R)$ is $n$ times G\^{a}teaux $\Sc^2$-differentiable at every bounded self-adjoint operator and, under the additional assumption ``$f^{(j)}$ is bounded, $j=0,\ldots,n$", at every 
self-adjoint operator.

We prove (see Theorem \ref{Cor1}) that $f$ is $n$ times continuously Fr\'{e}chet $\Sc^p$-differentiable at every bounded self-adjoint operator if and only if $f\in C^n(\R)$, generalizing the differentiability result of \cite{KPSS} for $n=1$ to an arbitrary natural $n$ and establishing a new property of continuous dependence of the differential on the point of differentiation. We obtain new sufficient conditions for an arbitrary order continuous Fr\'{e}chet differentiability at an unbounded self-adjoint operator that significantly extend analogous results of \cite{ACDS,KPSS}. Namely, we prove that $f$ is $n$ times continuously Fr\'{e}chet $\Sc^p$-differentiable at every self-adjoint operator if $f\in C^n(\R)$, all derivatives $f',\ldots,f^{(n)}$ are bounded, and $f^{(n)}\in C_0(\R)$ (see Theorem \ref{C0}) as well as that $f$ is $n$ times continuously Fr\'{e}chet $\Sc^p$-differentiable at every self-adjoint operator if $f\in C^{n+1}(\R)$ and all derivatives $f',\ldots,f^{(n+1)}$ are bounded (see Theorem \ref{C(n+1)}).
We also prove (see Theorem \ref{Thmi}(ii)) that $f$ is $n$ times G\^{a}teaux $\Sc^p$-differentiable under the relaxed assumption ``$f\in C^n(\R)$ and all derivatives $f',\ldots,f^{(n)}$ are bounded" (the latter property is a necessary condition for $n$ times G\^{a}teaux $\Sc^p$-differentiability, see Proposition \ref{nc}), extending the result of \cite{CLMSS} from $\Sc^2$ to the general $\Sc^p$ and the consequence of \cite{Peller2006} from $f\in B_{\infty1}^n(\R)\cap B_{\infty1}^1(\R)$ to $f\in C^n(\R)$ with bounded derivatives. Finally, we prove (see Theorem \ref{Th33Besov}) that
every $f\in B_{\infty1}^n(\R)\cap B_{\infty1}^1(\R)$ is $n$ times continuously Fr\'{e}chet $\Sc^q$-differentiable, $1\le q<\infty$, substantially strengthening G\^{a}teaux $\Sc^q$-differentiability that follows from \cite{Peller2006}.

Study of operator derivatives has been mainly motivated by development of perturbation theory. Initially operator derivatives were calculated in the operator norm and appeared in Taylor-like approximations of operator functions (see, for instance, \cite{ASsurvey} for details and references). It was necessary to involve higher order derivatives to produce approximations with respective remainders in symmetrically normed ideals. Replacing the first and second order derivatives of operator functions with respect to the operator norm by the $\Sc^2$-norm derivatives allowed to extend the fundamental results of \cite{Krein,Kop} on Taylor approximations to much broader sets of functions in \cite{Peller2016,CLMSS}. In Theorem \ref{corpert} we apply our main results to extend the estimate of \cite[Theorem 4.1]{PSST} for the $n$th order Taylor approximation with bounded initial self-adjoint operator and $f\in C^n(\R)$ to an unbounded initial operator and $f\in C^n(\R)$ with bounded $f',\ldots,f^{(n)}$.

Advancement in the study of operator smoothness and differentiability has been based on techniques known under the name ``multiple operator integration" and developing since \cite{DK1956}.  There are two principal approaches to multiple operator integration on Schatten classes (see, e.g., \cite{ST} for details and references). The approach in \cite{CLS,Peller2006} depends on separation of variables of a symbol and applies to the entire set of symbols treated in this paper only when $p=2$; the approach in \cite{PSS} builds on harmonic analysis of UMD spaces. To establish higher order Fr\'{e}chet differentiability and include unbounded operators we combine advantages of both approaches and perform multi-stage approximations on symbols, perturbations, and generating operators. Our method differs from the methods of \cite{CLS,Peller2006} for G\^{a}teaux differentiability resting on the former approach, which was sufficient due to the weaker type of differentiability and either restriction to a smaller set of functions or to perturbations in $\Sc^2$, as well as the method of \cite{KPSS} for the first order Fr\'{e}chet differentiability resting on the latter approach, which was sufficient due to a stronger technical machinery available in the first order case and restriction to bounded operators.

We fix the following notations and conventions to be used throughout the paper. Let $\Bc_n(X_1\times\cdots\times X_n,Y)$ denote the space of bounded $n$-linear operators mapping the Cartesian product $X_1\times\cdots\times X_n$ of Banach spaces $X_1,\ldots,X_n$ to a Banach space $Y$, let $\Hc$ denote a separable Hilbert space, $\Sc^p$ the $p$th Schatten-von Neumann class of compact operators on $\Hc$, $1\le p<\infty$, $\Sc^p_{sa}$ the subset of self-adjoint elements of $\Sc^p$, and $\tr$ the canonical trace on the ideal $\Sc^1$. For basic properties of Schatten ideals see, e.g., \cite{Simon}. When the domain of a self-adjoint operator is not specified, it is assumed to be defined either on $\Hc$ or on a dense subspace of $\Hc$. For any $n\in\N$,
let $C^n(\R)$ denote the linear space of $n$ times continuously differentiable functions on $\R$.
Let $C_b(\R)$ denote the space of continuous bounded functions on $\R$,
$C_0(\R)$ the space of continuous functions on $\R$ which tend to $0$ at
$\pm\infty$, and $\text{Lip}(\R)$ the space of Lipschitz functions on $\R$.
Let $\sym_n$ denote the group of all permutations of the set $\{1,\ldots,n\}$.
For any $f\in C^n(\R)$, let $f^{[n]}\colon\R^{n+1}\to\C$
denote the $n$th order divided
difference of $f$. We recall that it
is defined recursively as follows:
\begin{align*}
&f^{[0]}(x_1):=f(x_1),\\
&f^{[n]}(x_1,\dots,x_{n+1}):=
\lim_{x\to x_{n+1}}
\frac{f^{[n-1]}(x_1,\dots,x_{n-1},x)-f^{[n-1]}(x_1,\dots,x_{n-1},x_n)}{x-x_n}.
\end{align*}
We also recall that $f^{[n]}$ is bounded provided that $f^{(n)}\in C_b(\R)$.

\section{Multiple operator integration}
\label{s2}

In this section we collect existing and derive new properties of
multiple operator integrals that are crucial for proofs of our main results.

G\^{a}teaux and Fr\'{e}chet derivatives
of our operator functions will be represented via multiple operator
integrals introduced in \cite[Definition 3.1]{PSS}. We recall this definition below.

Let $A_j=A_j^*$, $j=1,\ldots,n+1$, and denote
\[E_{l,r}^j = E_{A_j}\left(\left[ \frac lr, \frac {l + 1}r\right)\right),\]
for every~$r \in \N$, $l \in \Z$, and $j=1,\ldots,n+1$. Let~$n \in \N$ and
$1 \le p,p_i < \infty$, $i=1,\ldots,n$. Let ~$X_i \in \Sc^{p_i}$ and let
$\varphi\colon \R^{n + 1} \to \C$ be a bounded Borel function.

\begin{definition}
\label{MOIPSS}
Suppose that for every~$r \in \N$, the series
\begin{align*}
S_{\varphi, r} (X_1, \ldots, X_n )
:&= \sum_{l_1, \ldots,l_{n+1} \in \Z} \varphi \left( \frac {l_1}r, \ldots, \frac
{l_{n+1}}r \right) \, E_{l_1, r}^1 X_1 E_{l_2, r}^2 X_2 \cdots X_n E_{l_{n+1}, r}^{n+1}\\
&=\lim\limits_{N\rightarrow\infty}\sum_{\substack{|l_j|\le N\\ 1\le j\le n+1}}\varphi
\left( \frac {l_1}r,\ldots,\frac{l_{n+1}}r\right)\,E_{l_1,r}^1 X_1 E_{l_2,r}^2 X_2\cdots
X_n E_{l_{n+1},r}^{n+1}
\end{align*}
converges in the norm of~$\Sc^p$ and
$$
(X_1, \ldots, X_n ) \mapsto S_{\varphi, r} (X_1, \ldots, X_n ),\ \ r \in \N,
$$
is a sequence of bounded multilinear operators that map $\Sc^{p_1} \times \cdots \times
\Sc^{p_n} \to \Sc^p$.  If the sequence of
operators~$\left\{S_{\varphi,r}\right\}_{r=1}^\infty$ converges strongly to some bounded multilinear operator $T\colon
\Sc^{p_1} \times \cdots \times \Sc^{p_n} \to \Sc^p$, then $T$
is called the multiple operator integral
associated with~$\varphi$ and the operators~$A_1,\ldots,A_{n+1}$
(or the spectral measures~$E_{A_1},\ldots,E_{A_{n+1}}$),
and $T$ is denoted by $T_{\varphi}^{A_1,\ldots,A_{n+1}}$.
\end{definition}

We note that there are other multiple operator integral constructions (see, e.g., \cite{ST}), but the one described in Definition \ref{MOIPSS} allows to handle $\varphi=f^{[n]}$ for the largest set of functions treated in this paper, that is, for $f\in C^n(\R)$ with $f^{(n)}\in C_b(\R)$.

The following conditions for existence of bounded multiple operator
integrals and estimates for their norms are crucial in our proofs.

\begin{theorem}
\label{ThPSS}
Let $1<p<\infty$, $n\in\N$, $f\in C^n(\R)$, $f^{(n)}\in C_b(\R)$.
Let $A_1,\ldots,A_{n+1}$ be self-adjoint operators and let $X_1,\ldots,X_n\in\Sc^p$.
Then, $T_{f^{[n]}}^{A_1,\ldots,A_{n+1}}\in\Bc_n(\Sc^{pn}\times\cdots\times\Sc^{pn},\Sc^p)$ and there exists
$c_{p,n}>0$ such that
\begin{align}
\label{ThPSSi}
\big\|T_{f^{[n]}}^{A_1,\ldots,A_{n+1}}(X_1,\ldots,X_n)\big\|_p
&\le c_{p,n}\,\|f^{(n)}\|_\infty\|X_1\|_{np}\cdots\|X_n\|_{np}\\
\label{ThPSSii}
&\le c_{p,n}\,\|f^{(n)}\|_\infty\|X_1\|_p\cdots\|X_n\|_p.
\end{align}
\end{theorem}

\begin{proof}
We note that the inequality \eqref{ThPSSii} follows from \eqref{ThPSSi} because $\norm{\,\cdotp}_{np}\le \norm{\,\cdotp}_{p}$ and focus on the proof of \eqref{ThPSSi}.

In the case $A_1=\cdots=A_n$, the inequality \eqref{ThPSSi}
is established in \cite[Theorem 5.3 and Remark 5.4]{PSS}.

Assume now that some of $A_1,\ldots,A_{n+1}$ are distinct.
A careful investigation of the proofs of
\cite[Lemmas 3.3 and 5.5]{PSS} shows that their results also hold
for distinct $A_1,\ldots,A_{n+1}$. This immediately implies that
to get \eqref{ThPSSi} in the general case it suffices to prove
\eqref{ThPSSi} for $A_1,\ldots,A_{n+1}$ whose spectra are finite subsets of $\Z$.


Assume now that the spectra of $A_1,\ldots,A_{n+1}$ are finite subsets of $\Z$. Then the transformation $T_{f^{[n]}}^{A_1,\ldots,A_{n+1}}(X_1,\ldots,X_n)$ is well defined and represented by a finite sum.
Let $1\leq k\leq n$ and assume that the
$(n-k+1)$ last self-adjoint operators $A_{k+1},A_{k+2},\ldots,
A_{n+1}$ are equal.
%
Let $e_{ij}\in M_2(\C)$ denote the elementary matrix whose nonzero entry has indices $(i,j)$.
Let $\Xc_l=X_l\otimes e_{22}$ for $l=1,\ldots, k-1$, $\Xc_{k}=X_k\otimes e_{21}$, and
$\Xc_l=X_l\otimes e_{11}$ for $l=k+1,\ldots,n$.
Then for any $l=1,\ldots,k$, let $\Ac_l$ be the self-adjoint operator with the spectral measure
$E_{A_{k+1}}\otimes e_{11}+E_{A_l}\otimes e_{22}$ and for any $l=k+1,\ldots,n+1$, let $\Ac_l=\Ac_{k}$.
Consider
$X\in\Sc^{p'}$, with $\frac1p+\frac{1}{p'}=1$, and let $\Xc=X\otimes e_{12}$.
A straightforward calculation (see, e.g., the proof of \cite[Theorem 3.3]{ASAIM})
implies that
$$
\tr\big(T_{f^{[n]}}^{A_1,\ldots,A_{n+1}}(X_1,\ldots,X_n)X\big)\,=\,
\tr\big(T_{f^{[n]}}^{\Ac_1,\ldots,\Ac_{n+1}}
(\Xc_1,\ldots,\Xc_n)\Xc\big).
$$
Note that by construction, the $(n-k+2)$ self-adjoint
operators $\Ac_{k},\Ac_{k+1},\ldots,
\Ac _{n+1}$ are equal.
Further $\norm{\Xc_l}_p=\norm{X_l}_p$ for any $l=1,\ldots,n$ and
$\norm{\Xc}_{p'}=\norm{X}_{p'}$.
Using this process inductively for $k=n, n-1,\ldots,1$, we obtain that to
prove \eqref{ThPSSi}, it suffices to have it when the
self-adjoint
operators are all equal. This concludes the proof.
\end{proof}

\begin{lemma}
\label{LAm}
Let $A=A^*$, $m\in\N$, $P_m=E_A((-m,m))$, $A_m=P_mA$, $K=K^*$ bounded,
$K_m=P_mKP_m$. Let $n\in\N, p>1$, $X_i\in\Sc^p$,
$X_{i,m}=P_mX_iP_m$, $i=1,\ldots,n$. Assume that $f\in C^n(\R), f^{(n)}\in C_b(\R)$. Then,
\begin{align}
\label{TmT}
T_{f^{[n]}}^{A_m+K_m,\ldots,A_m+K_m}(X_{1,m},\ldots,X_{n,m})
=T_{f^{[n]}}^{A+K_m,\ldots,A+K_m}(X_{1,m},\ldots,X_{n,m}).
\end{align}
\end{lemma}

\begin{proof}
Both multiple operator integrals in \eqref{TmT} are well defined
bounded multilinear transformations by Theorem \ref{ThPSS}.
Since $P_m(A+K_m)y=(A+K_m)P_my$ for every $y$ in the domain of $A$,
\begin{align*}
P_mE_{A+K_m}=E_{A+K_m}P_m=P_mE_{A+K_m}P_m
\end{align*}
by \cite[Lemma 5.6.17]{KR}.
Let $\Hc_m=P_m\Hc$, so $\Hc=\Hc_m\oplus\Hc_m^\perp$. By the spectral decomposition,
\begin{align*}
A_m+K_m=P_m(A+K_m)P_m=\int_\R\lambda\,d(P_m E_{A+K_m}(\lambda)P_m).
\end{align*}
We obtain the decomposition of the spectral measure of $A_m+K_m$,
\begin{align*}
E_{A_m+K_m}=P_mE_{A+K_m}P_m|_{\Hc_m}\oplus \delta_0\otimes I_{\Hc_m^\perp}.
\end{align*}
Hence,
\begin{align*}
P_m E_{A_m+K_m}=E_{A_m+K_m}P_m=P_m E_{A_m+K_m}P_m=P_m E_{A+K_m}P_m.
\end{align*}
Summarizing the observations made above we arrive at
\begin{align*}
&\lim_{r\rightarrow\infty}\sum_{l_1,l_2,\ldots,l_{n+1}\in\Z}f^{[n]}\bigg(\frac{l_1}{r},\ldots,\frac{l_{n+1}}{r}\bigg)
E_{A+K_m}\bigg(\bigg[\frac{l_1}{r},\frac{l_1+1}{r}\bigg)\bigg)P_mX_1P_m\ldots\\
&\quad\quad\ldots P_mX_nP_m E_{A+K_m}\bigg(\bigg[\frac{l_{n+1}}{r},\frac{l_{n+1}+1}{r}\bigg)\bigg)\\
&=\lim_{r\rightarrow\infty}\sum_{l_1,l_2,\ldots,l_{n+1}\in\Z}f^{[n]}\bigg(\frac{l_1}{r},\ldots,\frac{l_{n+1}}{r}\bigg)
E_{A_m+K_m}\bigg(\bigg[\frac{l_1}{r},\frac{l_1+1}{r}\bigg)\bigg)P_mX_1P_m\ldots\\
&\quad\quad\ldots P_mX_nP_m E_{A_m+K_m}\bigg(\bigg[\frac{l_{n+1}}{r},\frac{l_{n+1}+1}{r}\bigg)\bigg),
\end{align*}
which implies \eqref{TmT}.
\end{proof}

There is a different approach to multiple operator integrals $T_{f^{[n]}}^{A_1,\ldots,A_{n+1}}$ (see \cite{Peller2006}) that allows to estimate their $\Sc^1$-norm and operator norm, but it applies to a set of functions significantly smaller than the one in Theorem \ref{ThPSS}. By \cite[Lemma 3.5]{PSS}, the multiple operator integral constructed in \cite{Peller2006} coincides with the one presented in Definition \ref{MOIPSS} for $f$ in a certain Besov class, which definition is recalled below.

Let $w_0\in C^{\infty}(\mathbb{R})$ be such that its Fourier transform is supported in the set $[-2, -1/2] \cup [1/2, 2]$, $\hat{w_0}$ is an even function and $\hat{w_0}(y) + \hat{w_0}(y/2) = 1$ for $1 \le y \le 2$. Set
\begin{align}
\label{wk}
w_k(x) = 2^k w_0(2^kx)
\end{align}
for $x\in \mathbb{R}, k\in \mathbb{Z}$. Following \cite{Peller2006}, the Besov space is defined as the set
$$
B^n_{\infty 1}(\mathbb{R}) = \Big\lbrace f\in C^n(\mathbb{R}): \ \ \|f^{(n)}\|_{\infty} + \sum_{k\in \mathbb{Z}} 2^{nk} \|f \ast w_k \|_{\infty} < \infty \Big\rbrace,
$$
equipped with the seminorm
$$
\|f\|_{B^n_{\infty 1}} = \sum_{k\in \mathbb{Z}} 2^{nk} \|f \ast w_k \|_{\infty}.
$$

The following estimate for the $\Sc^q$-norm, $1\leq q <\infty$ of the transformation $T_{f^{[n]}}^{A_1,\ldots,A_{n+1}}$ is derived completely analogously to the estimate of $T_{f^{[n]}}^{A_1,\ldots,A_{n+1}}$ in the operator norm proved in \cite[Theorem 5.5]{Peller2006}.

\begin{theorem}
\label{Peller}
Let $n\in\N$ and $f\in B_{\infty1}^n(\mathbb{R})$. Let $1\leq q <\infty$.
Let $A_1,\ldots,A_{n+1}$ be self-adjoint operators and let $X_1,\ldots,X_n\in\Sc^q$.
Then, $T_{f^{[n]}}^{A_1,\ldots,A_{n+1}}\in\Bc_n(\Sc^{nq}\times\cdots\times\Sc^{nq},\Sc^q)$ and there exists
${\rm const}_n>0$ such that
\begin{align*}
\big\|T_{f^{[n]}}^{A_1,\ldots,A_{n+1}}(X_1,\ldots,X_n)\big\|_q
&\le {\rm const}_n\,\|f\|_{B_{\infty1}^n}\|X_1\|_{nq}\cdots\|X_n\|_{nq}\\
&\le {\rm const}_n\,\|f\|_{B_{\infty1}^n}\|X_1\|_q\cdots\|X_n\|_q.
\end{align*}
\end{theorem}

We will also need properties of multiple operator integrals that
follow from the approach of \cite{CLS}.
Again let $A_j=A_j^*$, $j=1,\ldots,n+1$ and let
$\lambda_{A_j}$  be a scalar-valued spectral measure for $A_j$ (that is, a measure
on the Borel subsets of $\sigma(A_j)$ having the sames sets of measure zero as $E_{A_j}$).
Let $\Omega=\prod_{j=1}^{n+1}\sigma(A_j)$ be the product measurable space. Then the tensor product space
$$
L^\infty(\sigma(A_1),\lambda_1)\otimes\cdots\otimes L^\infty(\sigma(A_{n+1}),\lambda_{n+1})
\,\subset\,L^\infty(\Omega,\lambda_{A_1}\times\cdots\times\lambda_{A_{n+1}})
$$
is a $w^*$-dense subspace.

Further $\Bc_n(\Sc^2\times\cdots\times\Sc^2,\Sc^2)$ is a dual space, namely it naturally
identifies with the $(n+1)$-fold projective tensor product of $\Sc^2$ (see \cite[Section 3.1]{CLS}).

\begin{definition}\label{MOICLS}
Let $\Gamma$ be the unique linear map from the tensor product space
$\linebreak
L^\infty(\sigma(A_1),\lambda_1) \otimes\cdots\otimes
L^\infty(\sigma(A_{n+1}),\lambda_{n+1})$ into
$\Bc_n(\Sc^2\times\cdots\times\Sc^2,\Sc^2)$ such that
$$
\Gamma(f_1\otimes\cdots\otimes f_{n+1})(X_1,\ldots,X_n)
=f_1(A_1)X_1f_2(A_2)X_2\cdots f_n(A_n)X_nf_{n+1}(A_{n+1})
$$
for all $f_j\in L^\infty(\sigma(A_j),\lambda_j), j=1,\ldots,n$,
and all $X_1,\ldots, X_n\in \Sc^2$. According to \cite[Proposition 6]{CLS},
$\Gamma$ uniquely extends to a $w^*$-continuous and contractive map
$$
\Gamma^{A_1,\ldots,A_{n+1}}\colon L^\infty(\Omega,\lambda_{A_1}\times\cdots\times\lambda_{A_{n+1}})
\longrightarrow \Bc_n(\Sc^2\times\cdots\times\Sc^2,\Sc^2).
$$
Let $\varphi\colon\R^{n+1}\to\C$ be a bounded Borel function
and let $\widetilde{\varphi}\in
L^\infty(\Omega,\lambda_{A_1}\times\cdots\times\lambda_{A_{n+1}})$ be the class of
its restriction to $\Omega$.
Then the $n$-linear map $\Gamma^{A_1,\ldots,A_{n+1}}(\widetilde{\varphi})$
will be simply denoted by
$$
\Gamma^{A_1,\ldots,A_{n+1}}(\varphi)\colon \Sc^2\times\cdots\times\Sc^2\longrightarrow\Sc^2
$$
in the sequel.
\end{definition}

According to \cite[Remark 8]{CLS}, the above multiple operator
integrals $\Gamma^{A_1,\ldots,A_{n+1}}(\varphi)$ coincide with
Pavlov's ones \cite{Pav}. The crucial point in the
construction leading to Definition \ref{MOICLS} is the $w^*$-continuity
of $\Gamma^{A_1,\ldots,A_{n+1}}$, which allows to reduce various computations
to elementary tensor product manipulations. See \cite{CLMSS} for
illustrations.

\begin{proposition}\label{Coincide0}
\label{L}
Let $n\in\N$, let $A_1,\ldots,A_{n+1}$ be self-adjoint operators
and $X_1,\ldots,X_n\in\Sc^2$. If $f\in C^n(\R), f^{(n)}\in C_b(\R)$, then
\begin{equation}\label{Coincide}
\big[\Gamma^{A_1,\ldots,A_{n+1}}(f^{[n]})\big](X_1,\ldots,X_n)=
T_{f^{[n]}}^{A_1,\ldots,A_{n+1}}(X_1,\ldots,X_n).
\end{equation}
\end{proposition}

\begin{proof}
By Theorem \ref{ThPSS}\eqref{ThPSSii},
$T_{f^{[n]}}^{A_1,\ldots,A_{n+1}}\colon
\Sc^2\times\cdots\times\Sc^2\to\Sc^2$ is well defined.

For any $r\in\N, l\in\Z$, set $J_{l,r}=\bigl[\frac{l}{r},\frac{l+1}{r}\bigr)$
and for $N\in\N$, consider
$$
f^{[n]}_{r,N} = \sum_{\substack{|l_j|\le N\\ 1\le j\le n+1}}
f^{[n]}\left( \frac {l_1}r, \ldots, \frac
{l_{n+1}}r \right) \chi_{J_{l_1,r}}\otimes\cdots\otimes \chi_{J_{l_{n+1},r}}.
$$
Since $f^{[n]}\colon\R^{n+1}\to\C$ is continuous,
$$
f^{[n]} = w^*\hbox{-}\lim_{r\to\infty}\lim_{N\to\infty} f^{[n]}_{r,N}
$$
in $L^\infty(\Omega,\lambda_{A_1}\times\cdots\times\lambda_{A_{n+1}})$.
Hence for any $X_1,\ldots,X_n\in\Sc^2$,
$$
\big[\Gamma^{A_1,\ldots,A_{n+1}}(f^{[n]})\big](X_1,\ldots,X_n)
=\lim_{r\to\infty}\lim_{N\to\infty}
\big[\Gamma^{A_1,\ldots,A_{n+1}}(f^{[n]}_{r,N})\big](X_1,\ldots,X_n)
$$
in $\Sc^2$. Comparing with Definition \ref{MOIPSS}, we deduce (\ref{Coincide}).
\end{proof}

\begin{remark}
\

\smallskip
(i) In the case when $p=2$, Theorem \ref{ThPSS}\eqref{ThPSSii}
has a simple proof and the constant appearing
in (\ref{ThPSSii}) is $c_{n,n}=1$. More generally it can be shown that
for any bounded continuous function $\varphi\colon\R^{n+1}\to\C$,
$T_{\varphi}^{A_1,\ldots,A_{n+1}}\colon
\Sc^2\times\cdots\times\Sc^2\to\Sc^2$ is well defined. Then the above proof
implies that $\big[\Gamma^{A_1,\ldots,A_{n+1}}(\varphi)\big](X_1,\ldots,X_n)=
T_{\varphi}^{A_1,\ldots,A_{n+1}}(X_1,\ldots,X_n)$
for any $X_1,\ldots,X_n\in \Sc^2$. These facts, which will not be used in this
paper, are left as an exercise for the reader.

\smallskip
(ii) Let $A_1,\ldots,A_{n+1}$ and $f$ be as in Proposition \ref{Coincide0} and
let $1<p<\infty$. Then the mapping $\Gamma^{A_1,\ldots,A_{n+1}}(f^{[n]})$
extends to a bounded $n$-linear map from $\Sc^{np}\times\cdots\times
\Sc^{np}$ into $\Sc^{p}$, and (\ref{Coincide}) holds true
for any $X_1,\ldots, X_n$ in $\Sc^{np}$. This follows from Theorem
\ref{ThPSS}\eqref{ThPSSi}, Proposition \ref{Coincide0} and the density
of $\Sc^2\cap \Sc^{np}$ in $\Sc^{np}$.

\end{remark}

\section{Differentiability in $\Sc^p$, $1<p<\infty$}
\label{s3}

In this section we prove our main results on differentiability of functions
of operators in $\Sc^p$-norms, $1<p<\infty$.

We start by defining G\^{a}teaux and Fr\'{e}chet differentiability of operator functions.
The first order G\^{a}teaux and Fr\'{e}chet differentiability as well as higher order
Fr\'{e}chet differentiability are standard concepts (see, e.g., \cite[Chapter I, Sections B and F]{Schwartz}).
In this paper we understand higher order G\^{a}teaux differentiability in the sense described below.

Let $A=A^*$, $f\in \text{Lip}(\R)$, $1<p<\infty$. By \cite{BS3,PS}, the function
\begin{align}\label{PS-Function}
\varphi_{A,X,f,p} \colon \R\longrightarrow \Sc^p,\quad
\varphi_{A,X,f,p}(t)=f(A+tX)-f(A),
\end{align}
is well defined for every $X\in\Sc_{sa}^p$.


\begin{definition} Let $n\in\N$.
A function $f\in \text{Lip}(\R)$ is said to be $n$
times G\^{a}teaux $\Sc^p$-differentiable at $A=A^*$ if
\begin{itemize}
\item [(i)]
The function $\varphi_{A,X,f,p}$ defined by (\ref{PS-Function}) is $n$ times differentiable at $0$.
\item [(ii)]
$X\mapsto\varphi_{A,X,f,p}^{(k)}(0):\Sc_{sa}^p\mapsto\Sc^p$ is a bounded homogeneous transformation
of order $k$ for any $k=1,\ldots,n$.
\end{itemize}

The transformation $X\mapsto\varphi_{A,X,f,p}^{(n)}(0)$ is called the
$n$th G\^{a}teaux derivative of $f$ at $A$ and denoted $D_{G,p}^n f(A)$.
\end{definition}

Let $A=A^*$. By a $\Sc_{sa}^p$-neighborhood of $A$, we mean the set
$A+\mathcal V=\{A+X\, :\, X\in \mathcal V\}$, where
$\mathcal V\subset \Sc_{sa}^p$ is a neighborhood of $0$. Elements of
$A+\mathcal V$ are possibly unbounded
self-adjoint operators.

\begin{definition}\label{Sp-Diff}
Let $n\in\N$.
A function $f\in \text{Lip}(\R)$ is said to be $n$
times Fr\'{e}chet $\Sc^p$-differentiable at $A=A^*$ if it is $n-1$
times Fr\'{e}chet $\Sc^p$-differentiable in a
$\Sc_{sa}^p$-neighborhood of $A$ and there is a $n$-linear bounded operator
\begin{align*}
D_p^n f(A)\in\Bc_n(\Sc^p\times\cdots\times\Sc^p,\Sc^p)
\end{align*}
satisfying
\begin{align}
\nonumber
&\big\|\big(D_p^{n-1}f(A+X)-D_p^{n-1}f(A)\big)(X_1,\ldots,X_{n-1})-D_p^n f(A)(X_1,\ldots,X_{n-1},X)\big\|_p\\
&=o(\|X\|_p)\|X_1\|_p\cdots\|X_{n-1}\|_p
\end{align}
as $\|X\|_p\rightarrow 0$, $X\in\Sc^{p}_{sa}$, for all $X_1,\ldots, X_{n-1}\in\Sc^p$.


We further say that $f$ is $n$ times continuously Fr\'echet $\Sc^p$-differentiable
at $A=A^*$ if it is $n$ times Fr\'echet $\Sc^p$-differentiable in a $\Sc^p_{sa}$-neighborhood
of $A$ and for small $X\in \Sc^p_{sa}$,
\begin{equation}\label{4n+1}
\bignorm{\bigl(D_p^{n}f(A+X) -D_p^nf(A)\bigr)(X_1,\ldots,X_n)}_p =o(1)\|X_1\|_p\cdots\|X_{n}\|_p
\end{equation}
as $\|X\|_p\rightarrow 0$,
for all $X_1,\ldots, X_{n}\in\Sc^p$.
\end{definition}

For $A=A^*$ and $f\in \text{Lip}(\R)$, define
\begin{align}\label{PS-Function2}
\psi_{A,f,p} \colon \Sc^{p}_{sa}\longrightarrow \Sc^p,\quad
\psi_{A,f,p}(X)=f(A+X)-f(A).
\end{align}
A thorough look at Definition \ref{Sp-Diff} shows that
$f$ is $n$ times Fr\'{e}chet $\Sc^p$-differentiable at $A=A^*$ (resp.
$n$ times continuously Fr\'echet $\Sc^p$-differentiable
at $A=A^*$) if and only if $\psi_{A,f,p}$ is $n$ times
Fr\'echet differentiable at $0$ (resp.  $n$ times continuously
Fr\'echet differentiable at $0$) in the usual sense of differential calculus.

Since $\varphi_{A,X,f,p}(t)=\psi_{A,f,p}(tX)$, it follows from standard  functional analysis that
if $f$ is $n$ times Fr\'{e}chet $\Sc^p$-differentiable at $A=A^*$, then it is also
$n$ times G\^{a}teaux $\Sc^p$-differentiable at $A$ and
\begin{align*}
D_{G,p}^n f(A)(X)=D_p^n f(A)(X,\ldots,X),\quad X\in\Sc^{p}_{sa}.
\end{align*}

If $A$ is bounded, then all the above definitions make sense if f is a locally Lipschitz function. Indeed, the definition of an operator $f(A+X)-f(A)$ only depends on the restrictions of $f$ to the spectra of $A$ and $A+X$, which are compact subsets of $\C$.

Our main results are stated below.

\begin{theorem}\label{C(n+1)}
\label{Thmiin+1}
Let $1<p<\infty$, $n\in\N$, and $f\in C^{n+1}(\R)$ satisfy
$f',\ldots,f^{(n+1)}\in C_b(\R)$. Then $f$ is $n$ times continuously Fr\'{e}chet $\Sc^p$-differentiable at every $A=A^*$ and
\begin{align}
\label{*2n+1}
D_p^k f(A)(X_1,\ldots,X_k)=\sum_{\sigma\in \sym_k}
T_{f^{[k]}}^{A,\ldots,A}(X_{\sigma(1)},\ldots,X_{\sigma(k)})
\end{align}
for every $k=1,\ldots,n$, and all $X_1,\ldots,X_k\in\Sc^p$.
\end{theorem}

\begin{theorem}\label{C0}
Let $1<p<\infty$, $n\in\N$. Let $f\in C^n(\R)$ satisfy $f',\ldots,f^{(n-1)}\in
C_b(\R)$ and $f^{(n)}\in C_0(\R)$. Then $f$ is $n$ times continuously Fr\'{e}chet
$\Sc^p$-differentiable at every $A=A^*$ and
\begin{align}
\label{*2}
D_p^k f(A)(X_1,\ldots,X_k)
=\sum_{\sigma\in \sym_k}T_{f^{[k]}}^{A,\ldots,A}(X_{\sigma(1)},\ldots,X_{\sigma(k)}),
\end{align}
for every $k=1,\ldots,n$, and all $X_1,\ldots,X_k\in\Sc^{p}$.
\end{theorem}

\begin{remark}\label{Sharp}
Under the assumptions of either Theorem \ref{C(n+1)} or Theorem \ref{C0}, we will actually prove the following stronger results.
\begin{itemize}
\item [(i)]
For every $k=1,\ldots,n$, and given $\varepsilon>0$,
there exists $\delta>0$ such that for all
$X_1,\ldots,X_{k-1}\in\Sc^p$ and for every $X_k\in \Sc^{p}_{sa}$
with $\norm{X_k}_{(k+1)p}\le\delta$,
\begin{align*}
\bignorm{\bigl(D^{k-1}_pf(A+X_k) &- D^{k-1}_pf(A)\bigr)(X_1,\ldots,X_{k-1})\\
-\sum_{\sigma\in{\rm Sym}_k}& T^{A,\ldots,A}_{f^{[k]}}\bigl(X_{\sigma(1)},\ldots,X_{\sigma(k)}\bigr)}_{p}
\le\,\varepsilon\norm{X_1}_{kp}\cdots\norm{X_k}_{kp}.
\end{align*}
\item [(ii)] Given $\varepsilon>0$,
there exists $\delta>0$ such that for every
$X\in\Sc_{sa}^p$ with $\norm{X}_{(n+1)p}\le\delta$
and for all $X_1,\ldots,X_n\in\Sc^p$,
$$
\bignorm{\bigl(D_p^{n}f(A+X) -D_p^nf(A)\bigr)(X_1,\ldots,X_n)}_p \le\,
\varepsilon\norm{X_1}_{np}\cdots\norm{X_n}_{np}.
$$
\end{itemize}
Since $\norm{\,\cdotp}_{kp}\le \norm{\,\cdotp}_{p}$ on $\Sc^p$, the results stated above imply Theorems \ref{C(n+1)} and \ref{C0}.
\end{remark}

We have the following strengthening of operator differentiability in the case of a bounded operator $A$.

\begin{theorem}
\label{Cor1}
Let $1<p<\infty$, $n\in\N$, and let $f:\R\to\C$ be a
locally Lipschitz function. Then $f$ is $n$ times continuously Fr\'{e}chet $\Sc^p$-differentiable at every bounded operator $A=A^*$ and \eqref{*2} holds if and only if $f\in C^n(\R)$.
\end{theorem}

We also establish $n$th order G\^{a}teaux differentiability of $f$
under relaxed assumptions on $f^{(n)}$.

\begin{theorem}
\label{Thmi}
Let $1<p<\infty$, $n\in\N$. Let $f\in C^n(\R)$ satisfy $f',\ldots,f^{(n)}\in C_b(\R)$.
Then $f$ is $n-1$ times continuously Fr\'{e}chet $\Sc^p$-differentiable and the following assertions hold.
\begin{enumerate}[(i)]
\item For any $X_1,\ldots, X_{n-1}\in\Sc^p$, any $X_n\in \Sc^{p}_{sa}$,
and any $\epsilon>0$, there exists $\delta>0$ such that $\vert t\vert<\delta$ implies
\begin{align}
\label{l3p}
\nonumber
&\bigg\|\big(D_p^{n-1}f(A+tX_n)-D_p^{n-1}f(A)\big)(X_1,\ldots,X_{n-1})-t\sum_{\sigma\in\sym_n}
T_{f^{[n]}}^{A,\ldots,A}(X_{\sigma(1)},\ldots,X_{\sigma(n)})\bigg\|_{p}\\
&\le\epsilon\, |t|\,\max_{1\le i\le n}\prod_{r\in\{1,\ldots,n\}\setminus\{i\}}\|X_r\|_{np}.
\end{align}
\item The function $f$ is $n$ times G\^{a}teaux $\Sc^p$-differentiable at every $A=A^*$,  with
\begin{align}
\label{Thmifla2}
D_{G,p}^n f(A)(X)=n!\,T_{f^{[n]}}^{A,\ldots,A}(X,\ldots,X)
\end{align}
for all
$X\in\Sc_{sa}^p$.
\end{enumerate}
\end{theorem}

Finally, as a consequence of Theorem \ref{Thmi}, we will obtain the following estimate
for operator Taylor remainders. It generalizes the analogous result of
\cite[Theorem 4.1]{PSST} from bounded to unbounded operators $A$.

\begin{theorem}\label{Remainders}
\label{corpert}
Let $1<p<\infty$, $n\in\N$  and $A=A^*$.
Let $f\in C^n(\R)$ satisfy $f',\ldots,f^{(n)}\in C_b(\R)$ and let $X\in\Sc^{np}_{sa}$.
Denote
\begin{align*}
R_{n,p,A,X,f}=f(A+X)-f(A)-\sum_{k=1}^{n-1}\frac{1}{k!}D_{G,p}^k\, f(A)(X,\ldots,X).
\end{align*}
Then,
\begin{align}
\label{*5}
\|R_{n,p,A,X,f}\|_p\le c_{p,n}\|f^{(n)}\|_\infty\|X\|_{np}^{n}.
\end{align}
\end{theorem}

The rest of this section is dedicated to the proofs of our main results.

The next proposition extends the result of \cite[Proposition 7.14]{KPSS} from $n=1$ to the case of a
general $n\in\N$.

\begin{proposition}
\label{nc}
Let $f\colon\R\to\C$ be a locally Lipschitz function, $1<p<\infty$, $n\in\N$.
Then, the following assertions hold.
\begin{enumerate}[(i)]
\item If $f$ is $n$ times G\^{a}teaux $\Sc^p$-differentiable at every bounded self-adjoint operator, then $f$ is $n$ times differentiable on $\R$ and $f',\ldots,f^{(n)}$ are bounded on compact subsets of $\R$.
Moreover, if $f\in \text{\rm Lip}(\R)$ is $n$ times G\^{a}teaux $\Sc^p$-differentiable at every self-adjoint operator, then $f',\ldots,f^{(n)}$ are bounded on $\R$.

\item If $f$ is $n$ times Fr\'{e}chet $\Sc^p$-differentiable at every bounded self-adjoint operator, then
$f\in C^n(\R)$.
\end{enumerate}
\end{proposition}

\begin{proof}
Let $\{Q_k\}_{k=1}^\infty$ be a sequence of mutually orthogonal
rank one orthogonal projections with
$sot\text{-}\sum_{k=1}^\infty Q_k=I$, and note $\|Q_k\|_p=1$.
Let $\{\la_k\}_{k=1}^\infty$ be a sequence in $\R$ and define
\begin{align*}
A=sot\text{-}\sum_{k=1}^\infty\la_k Q_k.
\end{align*}
We note that $A$ is a self-adjoint operator
satisfying $AQ_k=Q_kA=\la_k Q_k$. Furthermore, $\{\la_k\}_{k=1}^\infty$ is bounded
if and only if $A$ is bounded.

(i) Assume that $f$ is $n$ times G\^{a}teaux $\Sc^p$-differentiable at $A$.
For any $k\in\N$ and any $t\in\R$,
$$
f(A+tQ_k)-f(A)=\bigl(f(\la_k+t)-f(\la_k)\bigr)Q_k,
$$
hence
$$
f(\la_k+t)-f(\la_k) =\tr\bigl(f(A+tQ_k)-f(A)\bigr).
$$
By the composition rule, this implies that  $f$ is $n$ times differentiable on $\R$.

Moreover
for any $k\in\N$, for any $t\in\R$ and for any $m=1,\ldots,n$, we have
\begin{equation}\label{fmQk}
f^{(m)}(\lambda_k +t)Q_k=D_{G,p}^{m}f(A+tQ_k)(Q_k).
\end{equation}
Hence
\begin{align}
\label{bddseq}
\sup_{k\in\N}|f^{(m)}(\la_k)|=\sup_{k\in\N}\|D_{G,p}^m f(A)(Q_k)\|_p<\infty
\end{align}
for any $m=1,\ldots,n$.

Applying \eqref{bddseq} to the bounded sequences $\{\la_k\}_{k=1}^\infty$ implies
that $f',\ldots,f^{(n)}$ are bounded on compact subsets of $\R$ whenever $f$ is $n$
times G\^{a}teaux $\Sc^p$-differentiable at every bounded self-adjoint operator.
Applying \eqref{bddseq} to all sequences $\{\la_k\}_{k=1}^\infty$ implies that
$f',\ldots,f^{(n)}$ are bounded whenever $f$ is $n$ times G\^{a}teaux
$\Sc^p$-differentiable at every self-adjoint operator.

(ii) By part (i), $f$ is $n$ times differentiable. It follows from \eqref{fmQk} that
\begin{align*}
&\big\|\big(f^{(n-1)}(\la_k+t)-f^{(n-1)}(\la_k)-t f^{(n)}(\la_k)\big)Q_k\big\|_p\\
&=\big\|D_{G,p}^{n-1}f(A+tQ_k)(Q_k)-D_{G,p}^{n-1}f(A)(Q_k)-t D_{G,p}^{n}f(A)(Q_k)\big\|_p=o(t)\quad\text{as}\quad t\to 0,
\end{align*}
uniformly in $k$. Hence, given $\epsilon>0$, there exists $\delta>0$ such that
\begin{align}
\label{n-1eps}
\big|f^{(n-1)}(\la_k+t)-f^{(n-1)}(\la_k)-tf^{(n)}(\la_k)\big|<\epsilon |t|
\end{align}
whenever $|t|<\delta$ and $k\in\N$.

Fix $\lambda_1\in\R$ and let $\{\la_k\}_{k=1}^\infty$ be a sequence in $\R$ converging to $\la_1$.
Then, there exists $j_0\in\N$ such that for every $j\in\N$, $j\ge j_0$, we have $|\la_j-\la_1|<\delta$. Applying \eqref{n-1eps} with $k=1$ and $t=\la_j-\la_1$ implies
\begin{align*}
\big|f^{(n-1)}(\la_j)-f^{(n-1)}(\la_1)-(\la_j-\la_1)f^{(n)}(\la_1)\big|<\epsilon |\la_j-\la_1|
\end{align*}
and applying \eqref{n-1eps} with $k=j$ and $t=\la_1-\la_j$ implies
\begin{align*}
\big|f^{(n-1)}(\la_1)-f^{(n-1)}(\la_j)-(\la_1-\la_j)f^{(n)}(\la_j)\big|<\epsilon |\la_j-\la_1|.
\end{align*}
Therefore, $|f^{(n)}(\la_j)-f^{(n)}(\la_1)|<2\epsilon$ whenever $j\ge j_0$, implying
\begin{align*}
\lim_{j\rightarrow\infty} f^{(n)}(\la_j)=f^{(n)}(\la_1)
\end{align*}
for every sequence $\{\la_j\}_{j=1}^\infty$ converging to $\la_1$, for every $\la_1\in\R$.
Thus, $f^{(n)}$ is continuous.
\end{proof}

We continue with important technical lemmas. All operators
in the next statements are well defined thanks to Theorem \ref{ThPSS}.

\begin{lemma}
\label{L1}
Let $1<p<\infty$, $n\in\N$, $n\ge 2$, and $f\in C^n(\R)$, $f^{(n-1)},f^{(n)}\in C_b(\R)$.
Let $A_1,\ldots,A_{n-1},A,B$ be self-adjoint
operators with $B-A \in\Sc^p$, let $X_1,\ldots,X_{n-1}\in\Sc^p$.
Then, for every $i=1,\ldots,n$,
\begin{align}
\label{l2*}
\nonumber
&T_{f^{[n-1]}}^{A_1,\ldots,A_{i-1},A,A_i,\ldots,A_{n-1}}(X_1,\ldots,X_{n-1})
-T_{f^{[n-1]}}^{A_1,\ldots,A_{i-1},B,A_i,\ldots,A_{n-1}}(X_1,\ldots,X_{n-1})\\
&=T_{f^{[n]}}^{A_1,\ldots,A_{i-1},A,B,A_i,\ldots,A_{n-1}}
(X_1,\ldots,X_{i-1},A-B,X_i,\ldots,X_{n-1}).
\end{align}
\end{lemma}

\begin{proof}

If $p=2$, then \eqref{l2*} follows from \cite[Corollary 4.4]{CLMSS}
because the transformations $T$ and
$\Gamma$ given by Definitions \ref{MOIPSS} and \ref{MOICLS} coincide
on $\Sc^2\times\cdots\times \Sc^2$ (see Proposition \ref{L}).

If $1<p<2$, then $\Sc^p\subset\Sc^2$, so \eqref{l2*}
holds for all $X_1,\ldots,X_{n-1},A-B\in\Sc^p$.

Let $p>2$ and recall that $\Sc^2\subset \Sc^p$, with
\begin{align}
\label{p2}
\|\cdot\|_p\le\|\cdot\|_2.
\end{align}
Assume that $X_1,\ldots,X_{n-1}\in\Sc^2$ and $A-B\in\Sc^p$. Let $\{K_m\}_{m}
\subset\Sc^{2}_{sa}$ be such that
\begin{align*}
\|A-B-K_m\|_p\rightarrow 0\quad\text{as}\quad m\rightarrow\infty.
\end{align*}
For brevity, we introduce the notations
\begin{align}
\label{notation1}
\T_{f^{[r]}}^{{\bf A},i,F}:=T_{f^{[r]}}^{A_1,\ldots,A_{i-1},F,A_i,
\ldots,A_r}(X_1,\ldots,X_r)
\end{align}
and
\begin{align}
\label{notation2}
\T_{f^{[r]}}^{{\bf A},i,F,G}(D):=T_{f^{[r]}}^{A_1,\ldots,A_{i-1},F,G,A_i,
\ldots,A_{r-1}}(X_1,\ldots,X_{i-1},D,X_i,\ldots,X_{r-1}).
\end{align}
We have
\begin{align}
\label{l2p1}
\T_{f^{[n-1]}}^{{\bf A},i,A}-\T_{f^{[n-1]}}^{{\bf A},i,B}=\big(\T_{f^{[n-1]}}^{{\bf A},i,A-K_m}-\T_{f^{[n-1]}}^{{\bf A},i,B}\big)+\big(\T_{f^{[n-1]}}^{{\bf A},i,A}-\T_{f^{[n-1]}}^{{\bf A},i,A-K_m}\big).
\end{align}
Since $\{A-K_m\}_{m}$ resolvent strongly converges to $B$, by \cite[Proposition 3.1]{CLMSS} and \eqref{p2}, the first group of summands in \eqref{l2p1} satisfies
\begin{align}
\label{l2p2}
\Sc^p\text{-}\lim_{m\rightarrow\infty}\big(\T_{f^{[n-1]}}^{{\bf A},i,A-K_m}-\T_{f^{[n-1]}}^{{\bf A},i,B}\big)=0.
\end{align}
By \eqref{l2*} applied in $\Sc^2$ to the second group of summands in \eqref{l2p1},
\begin{align}
\label{l2p3}
\big(\T_{f^{[n-1]}}^{{\bf A},i,A}-\T_{f^{[n-1]}}^{{\bf A},i,A-K_m}\big)&=\T_{f^{[n]}}^{{\bf A},i,A,A-K_m}(K_m)\\
\nonumber
&=\T_{f^{[n]}}^{{\bf A},i,A,A-K_m}(K_m-(A-B))+\T_{f^{[n]}}^{{\bf A},i,A,A-K_m}(A-B).
\end{align}
By Theorem \ref{ThPSS}\eqref{ThPSSii},
\begin{align}
\label{l2p4}
\nonumber
&\big\|\T_{f^{[n]}}^{{\bf A},i,A,A-K_m}(K_m-(A-B))\big\|_p\\
&\le c_{p,n}\|f^{(n)}\|_\infty\|X_1\|_p
\cdots\|X_{n-1}\|_p\|A-B-K_m\|_p\rightarrow 0\quad\text{as}\quad m\rightarrow\infty.
\end{align}
By \cite[Proposition 3.1]{CLMSS} and \eqref{p2}, for every $L\in\Sc^2$,
\begin{align*}
\Sc^p\text{-}\lim_{m\rightarrow\infty}
&\big(\T_{f^{[n]}}^{{\bf A},i,A,A-K_m}(L)-\T_{f^{[n]}}^{{\bf A},i,A,B}(L)\big)=0,
\end{align*}
that is, given $\epsilon>0,L\in\Sc^2$, there exists $m_{\epsilon,L}\in\N$
such that for every natural $m\ge m_{\epsilon,L}$,
\begin{align}
\label{l2p5}
\big\|\T_{f^{[n]}}^{{\bf A},i,A,A-K_m}(L)-\T_{f^{[n]}}^{{\bf A},i,A,B}(L)\big\|_p<\epsilon.
\end{align}
Given $\epsilon>0$, let $L\in\Sc^2$ be such that
\begin{align*}
\|A-B-L\|_p<\frac{\epsilon}{2\, c_{np,n}\,\|f^{(n)}\|_\infty\|X_1\|_p\cdots\|X_{n-1}\|_p}.
\end{align*}
By Theorem \ref{ThPSS}\eqref{ThPSSii}, we obtain
\begin{align}
\label{l2p6}
\big\|\T_{f^{[n]}}^{{\bf A},i,A,A-K_m}(A-B-L)-\T_{f^{[n]}}^{{\bf A},i,A,B}(A-B-L)\big\|_p< \epsilon.
\end{align}
Combining \eqref{l2p5} and \eqref{l2p6} implies
\begin{align}
\label{l2p7}
\Sc^p\text{-}\lim_{m\rightarrow\infty}\big(\T_{f^{[n]}}^{{\bf A},i,A,A-K_m}-
\T_{f^{[n]}}^{{\bf A},i,A,B}\big)(A-B)=0.
\end{align}
Combining \eqref{l2p1}--\eqref{l2p4} and \eqref{l2p7} implies
\eqref{l2*} for $X_1,\ldots,X_{n-1}\in\Sc^2$ and $A-B\in\Sc^p$, $p>2$.

Approximating each $X_j$ by a sequence $\{L_{j,m}\}_m\subset\Sc^2$ in the $\Sc^p$-norm and passing to the limit in
\begin{align*}
&\big(T_{f^{[n-1]}}^{A_1,\ldots,A_{i-1},A,A_i,\ldots,A_{n-1}}
-T_{f^{[n-1]}}^{A_1,\ldots,A_{i-1},B,A_i,\ldots,A_{n-1}}\big)(L_{1,m},\ldots,L_{n-1,m})\\
&=T_{f^{[n]}}^{A_1,\ldots,A_{i-1},A,B,A_i,\ldots,A_{n-1}}(L_{1,m},\ldots,L_{i-1,m},A-B,L_{i,m},\ldots,L_{n-1,m})
\end{align*}
as $m\rightarrow\infty$ with use of the estimate in Theorem \ref{ThPSS} completes the proof of \eqref{l2*} in the full generality.
\end{proof}

A useful straightforward consequence of Lemma \ref{L1} is stated below.

We will frequently use the notation ${\bf A}_i=\underbrace{A,\ldots,A}_i$ for any operator $A$.

\begin{lemma}
\label{LMOId2}
Let $1<p<\infty$, $n\in\N$, $f\in C^n(\R)$, $f^{(n-1)},f^{(n)}\in C_b(\R)$.
Let $A,B$ be self-adjoint
operators with $A-B \in\Sc^p$, let $X_1,\ldots,X_{n-1}\in\Sc^p$.
Then,
\begin{align*}
&T_{f^{[n-1]}}^{A,\ldots,A}(X_1,\ldots,X_{n-1})-T_{f^{[n-1]}}^{B,\ldots,B}(X_1,\ldots,X_{n-1})\\
&=\sum_{i=1}^n T_{f^{[n]}}^{{\bf B}_{i-1},A,B,{\bf A}_{n-i}}(X_1,\ldots,X_{i-1},A-B,X_i,\ldots,X_{n-1}).
\end{align*}
\end{lemma}

\begin{proof}
We have
\begin{align*}
\big(T_{f^{[n-1]}}^{A,\ldots,A}-T_{f^{[n-1]}}^{B,\ldots,B}\big)(X_1,\ldots,X_{n-1})
=\sum_{i=1}^n\big(T_{f^{[n-1]}}^{{\bf B}_{i-1},
{\bf A}_{n-i-1}}-T_{f^{[n-1]}}^{{\bf B}_i,{\bf A}_{n-i}}\big)(X_1,\ldots,X_{n-1}),
\end{align*}
which along with Lemma \ref{L1} implies the result.
\end{proof}

\begin{proof}[Proof of Theorem \ref{C(n+1)}]
We prove this theorem by induction on $n$. The base of induction $n=1$ and the
induction step can be proved in a completely analogous way, so we demonstrate the latter and omit the former.

We show below that if the result holds for $n=k-1$, then it also holds for $n=k$.
Consider $X_1,\ldots,X_{k-1}\in \Sc^p$ and $X_k\in\Sc^{p}_{sa}$. Since the result holds for $k-1$, we have
\begin{align}
\label{mpii0}
\nonumber
S_1&:=D_p^{k-1}f(A+X_k)(X_1,\ldots,X_{k-1})-D_p^{k-1}f(A)(X_1,\ldots,X_{k-1})\\
&=\sum_{\tau\in \sym_{k-1}}\big(T_{f^{[k-1]}}^{A+X_k,\ldots,A+X_k}
-T_{f^{[k-1]}}^{A,\ldots,A}\big)(X_{\tau(1)},\ldots,X_{\tau(k-1)}).
\end{align}
By Lemma \ref{LMOId2}, for every $\tau\in \sym_{k-1}$,
\begin{align}
\label{mpii-1}
S_1=\sum_{\tau\in \sym_{k-1}}\sum_{i=1}^k T_{f^{[k]}}^{{\bf A}_{i-1},A+X_k,A,{\bf (A+X_k)}_{k-i}}(X_{\tau(1)},\ldots,X_{\tau(i-1)},X_k,X_{\tau(i)},\ldots,X_{\tau(k-1)}).
\end{align}
Note that
\begin{align}
\label{mpii-2}
\nonumber
S_2&:=\sum_{\sigma\in\sym_k}T_{f^{[k]}}^{A,\ldots,A}(X_{\sigma(1)},\ldots,X_{\sigma(k)})\\
&=\sum_{i=1}^k\sum_{\tau\in \sym_{k-1}} T_{f^{[k]}}^{A,\ldots,A}(X_{\tau(1)},\ldots,X_{\tau(i-1)},X_k,X_{\tau(i)},\ldots,X_{\tau(k-1)}).
\end{align}
Combining \eqref{mpii0}--\eqref{mpii-2} implies
\begin{align}
\label{mpii-3}
\|S_1-S_2\|_p\le\sum_{i=1}^k\sum_{\tau\in \sym_{k-1}} \big\|\big(&T_{f^{[k]}}^{{\bf A}_{i-1},A+X_k,A,{\bf (A+X_k)}_{k-i}}
-T_{f^{[k]}}^{A,\ldots,A}\big)\\
\nonumber
&(X_{\tau(1)},\ldots,X_{\tau(i-1)},X_k,X_{\tau(i)},\ldots,X_{\tau(k-1)})\big\|_p.
\end{align}
It follows from Lemma \ref{L1} that for every $i=1,\ldots,k$,
\begin{align}
\label{mpii-4}
\nonumber
&\big(T_{f^{[k]}}^{{\bf A}_{i-1},A+X_k,A,{\bf (A+X_k)}_{k-i}}-T_{f^{[k]}}^{A,\ldots,A}\big)
(X_{\tau(1)},\ldots,X_{\tau(i-1)},X_k,X_{\tau(i)},\ldots,X_{\tau(k-1)})\\
&=T_{f^{[k+1]}}^{{\bf A}_{i-1},A+X_k,A,A,{\bf (A+X_k)}_{k-i}}
(X_{\tau(1)},\ldots,X_{\tau(i-1)},X_k,X_k,X_{\tau(i)},\ldots,X_{\tau(k-1)})\\
\nonumber
&\quad+\big(T_{f^{[k]}}^{{\bf A}_{i+1},{\bf (A+X_k)}_{k-i}}-T_{f^{[k]}}^{A,\ldots,A}\big)
(X_{\tau(1)},\ldots,X_{\tau(i-1)},X_k,X_{\tau(i)},\ldots,X_{\tau(k-1)})\\
\nonumber
&=:S_3+S_4.
\end{align}
By a reasoning similar to the one in the proof of Lemma \ref{LMOId2},
\begin{align}
\label{mpii-5}
S_4=\sum_{j=i+2}^{k+1} &T_{f^{[k+1]}}^{{\bf A}_{j-1},A+X_k,A,{\bf (A+X_k)}_{k+1-j}}\\
\nonumber
&(X_{\tau(1)},\ldots,X_{\tau(i-1)},
X_k,X_{\tau(i)},\ldots,X_{\tau(j-1)},X_k,X_{\tau(j)},\ldots,X_{\tau(k-1)}).
\end{align}

Combining \eqref{mpii-4} and \eqref{mpii-5} and then applying Theorem \ref{ThPSS}
ensures that for every $i=1,\ldots,k$,
\begin{align}
\nonumber
&\big\|\big(T_{f^{[k]}}^{{\bf A}_{i-1},A+X_k,A,{\bf (A+X_k)}_{k-i}}-T_{f^{[k]}}^{A,\ldots,A}\big)
(X_{\tau(1)},\ldots,X_{\tau(i-1)},X_k,X_{\tau(i)},\ldots,X_{\tau(k-1)})\big\|_p\\
&\le  k\, c_{p,k+1}\,\|f^{(k+1)}\|_\infty\|X_k\|_{(k+1)p}^2\,\|X_1\|_{(k+1)p}\cdots\|X_{k-1}\|_{(k+1)p}.
\end{align}
Combining the latter with \eqref{mpii-3} implies
\begin{align}
\label{mpii-6}
\nonumber
\bigg\|&D_p^{k-1}f(A+X_k)(X_1,\ldots,X_{k-1})-D_p^{k-1}f(A)(X_1,\ldots,X_{k-1})\\
&-\sum_{\sigma\in \sym_k}T_{f^{[k]}}^{A,\ldots,A}(X_{\sigma(1)},\ldots,X_{\sigma(k)})\bigg\|_p
=o(\|X_k\|_{(k+1)p})\,\|X_1\|_{(k+1)p}\cdots\|X_{k-1}\|_{(k+1)p}
\end{align}
as $\|X_k\|_{(k+1)p}\rightarrow 0$.
Hence, $f$ is $k$ times Fr\'{e}chet $\Sc^p$-differentiable at $A$ and \eqref{*2n+1} holds.
By the principal of mathematical induction we obtain that $f$ satisfies Remark \ref{Sharp}(i), that
$f$ is $n$ times Fr\'{e}chet $\Sc^p$-differentiable at $A$
and that \eqref{*2n+1} holds for every $k=1,\ldots,n$.

To prove Remark \ref{Sharp}(ii), and hence the continuity property (\ref{4n+1}),
we apply \eqref{*2n+1} with $k=n$, Lemma \ref{LMOId2}, and Theorem \ref{ThPSS}.
For any $X_1,\ldots, X_n\in\Sc^p$ and for small $X\in\Sc^{p}_{sa}$, we have
\begin{align}
\label{forp18}
\nonumber
&\big\|D_p^n f(A+X)(X_1,\ldots,X_n)-D_p^n f(A)(X_1,\ldots,X_n)\big\|_p\\
&=\bigg\|\sum_{\sigma\in \sym_n}\big(T_{f^{[n]}}^{A+X,\ldots,A+X}-T_{f^{[n]}}^{A,\ldots,A}\big)
(X_{\sigma(1)},\ldots,X_{\sigma(n)})\bigg\|_p\\
\nonumber
&=\bigg\|\sum_{\sigma\in \sym_n}\sum_{i=1}^{n+1} T_{f^{[n+1]}}^{{\bf A}_{i-1},A+X,A,
{\bf (A+X)}_{n-i}}
(X_{\sigma(1)},\ldots,X_{\sigma(i-1)},X,X_{\sigma(i)},\ldots,X_{\sigma(n)})\bigg\|_p\\
\nonumber
&\le (n+1)\,n!\,c_{p,n+1}\,\|f^{(n+1)}\|_\infty\|X\|_{(n+1)p}
\|X_1\|_{(n+1)p}\cdots\|X_n\|_{(n+1)p},
\end{align}
which yields the result.
\end{proof}

The following lemma on continuity of a multiple operator integral is
the last auxiliary result needed to prove Theorem \ref{C0}.

\begin{lemma}
\label{L2}
Let $1<p<\infty$, $n\in\N$, $1\le i\le n$. Let $f\in C^n(\R)$, $f^{(n)}\in C_0(\R)$. Then, for every $\epsilon>0$
there exists $\delta>0$ such that for all self-adjoint $A,A_1,\ldots,A_n$ and all
$X_1,\ldots,X_n\in\Sc^p$ the estimate
\begin{align}
\label{L2ineq}
\nonumber
&\big\|T_{f^{[n]}}^{A_1,\ldots,A_{i-1},A+X,A_i,\ldots,A_n}(X_1,\ldots,X_n)
-T_{f^{[n]}}^{A_1,\ldots,A_{i-1},A,A_i,\ldots,A_n}(X_1,\ldots,X_n)\big\|_{p}\\
&\le\epsilon\,\|X_1\|_{np}\cdots\|X_n\|_{np}
\end{align}
holds whenever $X\in\Sc^{p}_{sa}$ satisfies $\|X\|_{(n+1)p}<\delta$.
\end{lemma}

\begin{proof}
In this proof we adopt the notations \eqref{notation1} and \eqref{notation2}.

Assume first that $f^{(n)}\in C_c^1(\R)$, so that $f^{(n+1)}\in C_b(\R)$. By Lemma \ref{L1},
\begin{align}
\label{L1cor}
\T_{f^{[n]}}^{{\bf A},i,A+X}-\T_{f^{[n]}}^{{\bf A},i,A}
=\T_{f^{[n+1]}}^{{\bf A},i,A+X,A}(X).
\end{align}
By the inequality $\|\cdot\|_{(n+1)p}\le\|\cdot\|_{np}$,
representation (\ref{L1cor}) and Theorem \ref{ThPSS},
\begin{align}
\label{L21}
\nonumber
&\big\|\T_{f^{[n]}}^{{\bf A},i,A+X}-\T_{f^{[n]}}^{{\bf A},i,A}\big\|_{p}\\
\nonumber
&\le c_{p,n+1}\,\|f^{(n+1)}\|_\infty\|X\|_{(n+1)p}\|X_1\|_{(n+1)p}\cdots\|X_n\|_{(n+1)p}\\
&\le c_{p,n+1}\,\|f^{(n+1)}\|_\infty\|X\|_{(n+1)p}\|X_1\|_{np}\cdots\|X_n\|_{np}.
\end{align}

Assume now that $f\in C^n(\R)$, $f^{(n)}\in C_0(\R)$.
Given $\epsilon>0$, there exists $f_\epsilon\in C^{n+1}(\R)$
such that $f_\epsilon^{(n)}\in C_c^1(\R)$ and
\begin{align}
\label{f-fepsilon}
\|f^{(n)}-f_\epsilon^{(n)}\|_\infty<\frac{\epsilon}{4\,c_{p,n}}.
\end{align}
We have
\begin{align}
\label{L22}
\big\|\T_{f^{[n]}}^{{\bf A},i,A+X}-\T_{f^{[n]}}^{{\bf A},i,A}\big\|_{p}
\le\big\|T_{(f-f_\epsilon)^{[n]}}^{{\bf A},i,A+X}\big\|_{p}
+\big\|T_{(f-f_\epsilon)^{[n]}}^{{\bf A},i,A}\big\|_{p}
+\big\|T_{f_\epsilon^{[n]}}^{{\bf A},i,A+X}
-T_{f_\epsilon^{[n]}}^{{\bf A},i,A}\big\|_{p}.
\end{align}
By Theorem \ref{ThPSS} and \eqref{f-fepsilon},
\begin{align}
\label{L23}
\nonumber
\big\|T_{(f-f_\epsilon)^{[n]}}^{{\bf A},i,A+X}\big\|_{p}
+\big\|T_{(f-f_\epsilon)^{[n]}}^{{\bf A},i,A}\big\|_{p}
&\le 2\, c_{p,n}\,\|f^{(n)}-f_\epsilon^{(n)}\|_\infty\|X_1\|_{np}\cdots\|X_n\|_{np}\\
&\le\frac12\,\epsilon\,\|X_1\|_{np}\cdots\|X_n\|_{np}.
\end{align}
Combining \eqref{L21} for $f=f_\epsilon$, \eqref{L22}, \eqref{L23} guarantees that if
\begin{align*}
\|X\|_{(n+1)p}<\delta=\frac{\epsilon}{2\,c_{p,n+1}\,\|f_\epsilon^{(n+1)}\|_\infty}\,,
\end{align*}
then \eqref{L2ineq} holds.
\end{proof}


\begin{proof}[Proof of Theorem \ref{C0}]
It follows from Theorem \ref{Thmiin+1} that $f$ is at least $n-1$ times Fr\'{e}chet $\Sc^p$-differentiable
at every $A=A^*$ and \eqref{*2} holds for $k=1,\ldots,n-1$.

Our goal is to show that the statement of Remark \ref{Sharp}(i) holds, that is
\begin{align}
\nonumber
\label{mpii1pn}
\bigg\|&\big(D_p^{n-1}f(A+X_n)-D_p^{n-1}f(A)\big)(X_1,\ldots,X_{n-1})
-\sum_{\sigma\in\sym_n}T_{f^{[n]}}^{A,\ldots,A}(X_{\sigma(1)},\ldots,X_{\sigma(n)})\bigg\|_p\\
&=o(\|X_n\|_{np})\,\|X_1\|_{np}\cdots\|X_{n-1}\|_{np}
\end{align}
as $\|X_{n}\|_{(n+1)p}\rightarrow 0$, $X_n\in\Sc^{p}_{sa}$,
for all $X_1,\ldots, X_{n-1}\in\Sc^p$.

Combining \eqref{mpii0} and \eqref{mpii-1} for $k=n$ gives
\begin{align}
\label{mpii4}
\nonumber
&\big(D_p^{n-1}f(A+X_n)-D_p^{n-1}f(A)\big)(X_1,\ldots,X_{n-1})-\sum_{\sigma\in\sym_n}T_{f^{[n]}}^{A,\ldots,A}(X_{\sigma(1)},\ldots,X_{\sigma(n)})\\
&=\sum_{\tau\in\sym_{n-1}}\sum_{i=1}^n\big(T_{f^{[n]}}^{{\bf A}_{i-1},A+X_n,A,{\bf (A+X_n)}_{n-i}}-T_{f^{[n]}}^{A,\ldots,A}\big)\\
\nonumber
&\quad\quad\quad(X_{\tau(1)},\ldots,X_{\tau(i-1)},X_n,X_{\tau(i)},\ldots,X_{\tau(n-1)}).
\end{align}
Since
\begin{align*}
T_{f^{[n]}}^{{\bf A}_{i-1},A+X_n,A,{\bf (A+X_n)}_{n-i}}-T_{f^{[n]}}^{A,\ldots,A}&=T_{f^{[n]}}^{{\bf A}_{i-1},A+X_n,A,{\bf (A+X_n)}_{n-i}}-T_{f^{[n]}}^{{\bf A}_{i+1},{\bf (A+X_n)}_{n-i}}\\
&\quad+\sum_{j=i+1}^n \big(T_{f^{[n]}}^{{\bf A}_j,{\bf (A+X_n)}_{n+1-j}}-T_{f^{[n]}}^{{\bf A}_{j+1},{\bf (A+X_n)}_{n-j}}\big)
\end{align*}
for every $i=1,\ldots,n$, it follows from \eqref{mpii4} and Lemma \ref{L2} that \eqref{mpii1pn} holds.

To prove the continuity property stated in Remark \ref{Sharp}(ii)
(and hence to prove (\ref{4n+1})), we recall \eqref{forp18} and note
\begin{align*}
\eqref{forp18}
=\Big\|\sum_{\sigma\in\sym_n}\sum_{i=1}^{n+1}
\big(T_{f^{[n]}}^{{\bf A}_{i-1},{\bf (A+X)}_{n+2-i}}-T_{f^{[n]}}^{{\bf A}_i,
{\bf (A+X)}_{n+1-i}}\big)(X_{\sigma(1)},\ldots,X_{\sigma(n)})\Big\|_p.
\end{align*}
Then by Lemma \ref{L2},
$$
\Bignorm{\big(D_p^n f(A+X)-D_p^n f(A)\big)(X_1,\ldots,X_n)}_p
=o(1)\norm{X_1}_{np}\cdots\norm{X_n}_{np}
$$
as $\norm{X}_{(n+1)p}\to 0$, $X\in\Sc^{p}_{sa}$, for all $X_1,\ldots,X_n\in\Sc^p$, yielding the result.
\end{proof}

\begin{proof}[Proof of Theorem \ref{Cor1}]
The necessary condition is proved in Proposition \ref{nc}.
For any $f\in C^n(\R)$ and for any bounded interval $I\subset\R$, there exists
a function $\widetilde{f}\in C^n(\R)$ with compact support such that $f$ and
$\widetilde{f}$ coincide on $I$. Since the definition of an operator
$f(A+X)-f(A)$ only depends on the restrictions of $f$ to the spectra of $A$ and $A+X$,
the sufficient condition is an immediate consequence of Theorem \ref{C0}.
\end{proof}

\begin{proof}[Proof of Theorem \ref{Thmi}]
The fact that $f$ is $(n-1)$ times continuously Fr\'echet differentiable
follows from Theorem \ref{C(n+1)}.
The G\^{a}teaux differentiability in the case $n=1$ is proved in \cite[Theorem 7.18]{KPSS}.

Furthermore, (ii) follows from (i) applied with $X_1=\cdots=X_{n-1}=X_n=X\in\Sc^{p}_{sa}$.
Hence it suffices to establish (i).

Let $n>1$. We fix $X_1,\ldots,X_{n-1}\in\Sc^p$ and $X_n\in\Sc^{p}_{sa}$ and we let $\epsilon>0$.
Let $A_m,X_{1,m},\ldots,X_{n,m}$ be as in Lemma \ref{LAm}. By Theorem \ref{Cor1}
and Remark \ref{Sharp}, $f$ is $n$ times Fr\'{e}chet $\Sc^p$-differentiable at $A_m$ and  given $m\in\N$, there exists
$\delta_{m,\epsilon}>0$ such that $|t|<\delta_{m,\epsilon}$ implies
\begin{align}
\label{l3p1}
\nonumber
&\bigg\|\big(D_p^{n-1}f(A_m+tX_{n,m})-D_p^{n-1}f(A_m)\big)(X_{1,m},\ldots,X_{n-1,m})\\
&\quad-t\sum_{\sigma\in\sym_n}
T_{f^{[n]}}^{A_m,\ldots,A_m}(X_{\sigma(1),m},\ldots,X_{\sigma(n),m})\bigg\|_{p}
\le\frac{\epsilon}{4}\, |t|\,\|X_{1,m}\|_{np}\cdots \|X_{n-1,m}\|_{np}.
\end{align}
We have
\begin{align}
\label{l3p2}
\nonumber
&\big(D_p^{n-1}f(A+tX_n)-D_p^{n-1}f(A)\big)(X_1,\ldots,X_{n-1})-t\sum_{\sigma\in\sym_n}
T_{f^{[n]}}^{A,\ldots,A}(X_{\sigma(1)},\ldots,X_{\sigma(n)})\\
&=\big(D_p^{n-1}f(A_m+tX_{n,m})-D_p^{n-1}f(A_m)\big)(X_{1,m},\ldots,X_{n-1,m})\\
\nonumber
&\quad-t\sum_{\sigma\in\sym_n}T_{f^{[n]}}^{A_m,\ldots,A_m}(X_{\sigma(1),m},\ldots,X_{\sigma(n),m})\\
\nonumber
&+D_p^{n-1}f(A+tX_n)(X_1,\ldots,X_{n-1})-D_p^{n-1}f(A_m+tX_{n,m})(X_{1,m},\ldots,X_{n-1,m})\\
\nonumber
&+D_p^{n-1}f(A_m)(X_{1,m},\ldots,X_{n-1,m})-D_p^{n-1}f(A)(X_1,\ldots,X_{n-1})\\
\nonumber
&+t\sum_{\sigma\in\sym_n}\big(T_{f^{[n]}}^{A_m,\ldots,A_m}(X_{\sigma(1),m},\ldots,X_{\sigma(n),m})
-T_{f^{[n]}}^{A,\ldots,A}(X_{\sigma(1)},\ldots,X_{\sigma(n)})\big)\\
\nonumber
&=:[1]-[2]+[3]+[4]+[5].
\end{align}
The first two lines $[1]-[2]$ are treated by (\ref{l3p1}).
By (\ref{*2n+1}) and Lemma \ref{LAm},
\begin{align*}
D_p^{n-1}f(A_m+tX_{n,m})(X_{1,m},\ldots,X_{n-1,m})=D_p^{n-1}f(A+tX_{n,m})(X_{1,m},\ldots,X_{n-1,m}).
\end{align*}
Hence,
\begin{align}
\label{l3p3}
[3]&=D_p^{n-1}f(A+tX_n)(X_1,\ldots,X_{n-1})-D_p^{n-1}f(A+tX_{n,m})(X_1,\ldots,X_{n-1})\\
\nonumber
&+D_p^{n-1}f(A+tX_{n,m})(X_1,\ldots,X_{n-1})-D_p^{n-1}f(A+tX_{n,m})(X_{1,m},\ldots,X_{n-1,m}).
\end{align}
By (\ref{*2n+1}) and Lemma \ref{LMOId2}, the first difference in \eqref{l3p3} equals
\begin{align}
\label{l3p4}
&t\sum_{\sigma\in\sym_{n-1}}\sum_{j=1}^n
T_{f^{[n]}}^{{\bf (A+tX_{n,m})}_{j-1},A+tX_n,A+tX_{n,m},{\bf (A+tX_n)}_{n-j}}\\
\nonumber
&\quad\quad (X_{\sigma(1)},\ldots,X_{\sigma(j-1)},X_n-X_{n,m},X_{\sigma(j)},\ldots,X_{\sigma(n-1)}).
\end{align}
Likewise, the second difference in \eqref{l3p3} equals
\begin{align}
\label{l3p5}
&\sum_{\sigma\in\sym_{n-1}}\sum_{i=1}^{n-1}
T_{f^{[n-1]}}^{A+tX_{n,m},\ldots,A+tX_{n,m}}\\
\nonumber
&\quad\quad (X_{\sigma(1),m},\ldots,X_{\sigma(i-1),m},
X_{\sigma(i)}-X_{\sigma(i),m},X_{\sigma(i+1)},\ldots,X_{\sigma(n-1)}).
\end{align}
By letting $t=0$ in \eqref{l3p3}--\eqref{l3p5} we obtain 
\begin{align}
\label{l3p6}
&[4]\\
\nonumber
&=-\sum_{\sigma\in\sym_{n-1}}\sum_{i=1}^{n-1}
T_{f^{[n-1]}}^{A,\ldots,A}(X_{\sigma(1),m},\ldots,X_{\sigma(i-1),m},X_{\sigma(i)}-X_{\sigma(i),m},X_{\sigma(i+1)},\ldots,X_{\sigma(n-1)}).
\end{align}
Combining \eqref{l3p3}--\eqref{l3p6} implies
\begin{align}
\label{l3p7}
[3]+[4]&=t\sum_{\sigma\in\sym_{n-1}}\sum_{j=1}^n
T_{f^{[n]}}^{{\bf (A+tX_{n,m})}_{j-1},A+tX_n,A+tX_{n,m},{\bf (A+tX_n)}_{n-j}}\\
\nonumber
&\quad\quad (X_{\sigma(1)},\ldots,X_{\sigma(j-1)},X_n-X_{n,m},X_{\sigma(j)},\ldots,X_{\sigma(n-1)})\\
\nonumber
&\quad+\sum_{\sigma\in\sym_{n-1}}\sum_{i=1}^{n-1}
\big(T_{f^{[n-1]}}^{A+tX_{n,m},\ldots,A+tX_{n,m}}-T_{f^{[n-1]}}^{A,\ldots,A}\big)\\
\nonumber
&\quad\quad (X_{\sigma(1),m},\ldots,X_{\sigma(i-1),m},X_{\sigma(i)}
-X_{\sigma(i),m},X_{\sigma(i+1)},\ldots,X_{\sigma(n-1)})\\
\nonumber
&=:[6]+[7].
\end{align}
By Theorem \ref{ThPSS}\eqref{ThPSSi}, 
\begin{align}
\label{l3p8}
\big\|[6]\big\|_{p}
\le |t|\,n!\,c_{p,n}\,\|X_n-X_{n,m}\|_{np}\|X_1\|_{np}\cdots\|X_{n-1}\|_{np}.
\end{align}
To treat $[7]$,
we note that by Lemma \ref{LMOId2}, for every $\sigma\in\sym_{n-1}$ and every $i=1,\ldots,n-1$,
\begin{align*}
&\big(T_{f^{[n-1]}}^{A+tX_{n,m},\ldots,A+tX_{n,m}}-T_{f^{[n-1]}}^{A,\ldots,A}\big)\\
&\quad\quad (X_{\sigma(1),m},\ldots,X_{\sigma(i-1),m},X_{\sigma(i)}-X_{\sigma(i),m},X_{\sigma(i+1)},\ldots,X_{\sigma(n-1)})\\
&=\sum_{j=1}^n T_{f^{[n]}}^{{\bf A}_{j-1},A+tX_{n,m},A,{\bf (A+tX_{n,m})}_{n-j}}(Y_{i,j,1},\ldots,Y_{i,j,n}),
\end{align*}
where, for any $j=1,\ldots,n$, $(Y_{i,j,1},\ldots, Y_{i,j,n})$ is a permutation of the $n$-tuple
$$
\bigl(X_{\sigma(1),m},\ldots, X_{\sigma(i-1),m},
X_{\sigma(i)}-X_{\sigma(i),m},X_{\sigma(i+1)},\ldots, X_{\sigma(n-1)},tX_{n,m}\bigr).
$$

Hence, by Theorem \ref{ThPSS}\eqref{ThPSSi} and by $\|X_{r,m}\|_{np}\le\|X_r\|_{np}$,
\begin{align}
\label{l3p9}
\big\|[7]\big\|_p\le|t|\,c_{p,n}\,\sum_{i=1}^{n-1}\|X_i-X_{i,m}\|_{np}
\prod_{r\in\{1,\ldots,n\}\setminus\{i\}}\|X_r\|_{np}.
\end{align}
Let $m\in\N$ be such that
\begin{align}
\label{mchoice}
\max_{1\le i\le n}\|X_i-X_{i,m}\|_p<\frac{\epsilon}{2n\,n!\,c_{p,n}}.
\end{align}
Combining \eqref{l3p7}--\eqref{l3p9} implies
\begin{align}
\label{l3p10}
\big\|[3]+[4]\big\|_p<\frac{\epsilon}{2}\,|t|\,\max_{1\le i\le n}\prod_{r\in\{1,\ldots,n\}\setminus\{i\}}\|X_r\|_{np}.
\end{align}
We now deal with $[5]$.
Applying Lemma \ref{LAm} gives
\begin{align*}
&T_{f^{[n]}}^{A_m,\ldots,A_m}(X_{\sigma(1),m},\ldots,X_{\sigma(n),m})
-T_{f^{[n]}}^{A,\ldots,A}(X_{\sigma(1)},\ldots,X_{\sigma(n)})\\
&=T_{f^{[n]}}^{A,\ldots,A}(X_{\sigma(1),m},\ldots,X_{\sigma(n),m})
-T_{f^{[n]}}^{A,\ldots,A}(X_{\sigma(1)},\ldots,X_{\sigma(n)})\\
&=\sum_{i=1}^n T_{f^{[n]}}^{A,\ldots,A}
(X_{\sigma(1)},\ldots,X_{\sigma(i-1)},X_{\sigma(i),m}-X_{\sigma(i)},X_{\sigma(i+1),m},\ldots,X_{\sigma(n),m}).
\end{align*}
for every $\sigma\in\sym_n$. Hence, by Theorem \ref{ThPSS}\eqref{ThPSSi},
\begin{align*}
\big\|[5]\big\|_{p}\le|t|\,n\,n!\,c_{np,n}\,\max_{1\le i\le n}\|X_i-X_{i,m}\|_{np}\,
\max_{1\le i\le n}\prod_{r\in\{1,\ldots,n\}\setminus\{i\}}\|X_r\|_{np}.
\end{align*}
By taking $m$ as in \eqref{mchoice}, we deduce from the latter that
\begin{align}
\label{l3p11}
\big\|[5]\big\|_{p}<\frac{\epsilon}{2}\,|t|\,\max_{1\le i\le n}\prod_{r\in\{1,\ldots,n\}\setminus\{i\}}\|X_r\|_{np}.
\end{align}
Let $\delta_{\epsilon,m}$ be chosen as in \eqref{l3p1}, where $m$ satisfies \eqref{mchoice}. It follows from
\eqref{l3p1}, \eqref{l3p2}, \eqref{l3p10}, \eqref{l3p11} that if $|t|<\delta_{\epsilon,m}$, then \eqref{l3p} holds.
\end{proof}

\begin{proof}[Proof of Theorem \ref{Remainders}]
Existence of the derivatives $D_{G,p}^k\, f(A)(X,\ldots,X)$ is justified by Theorem \ref{Thmi}.
By Lemma \ref{L1} and induction on $n$, we obtain
\begin{align}
R_{n,p,A,X,f}=\big(T_{f^{[n-1]}}^{A+X,A,\ldots,A}-T_{f^{[n-1]}}^{A,\ldots,A}\big)(X,\ldots,X)
=T_{f^{[n]}}^{A+X,A,\ldots,A}(X,\ldots,X).
\end{align}
The estimate \eqref{*5} follows from the latter representation and Theorem \ref{ThPSS}.
\end{proof}

\section{Differentiability in $\Sc^q$, $1\le q<\infty$}
\label{s4}

In this section we demonstrate that the functions proved to be G\^{a}teaux differentiable with respect to the operator norm in \cite{Peller2006} are continuously Fr\'{e}chet differentiable with respect to the  $\Sc^q$-norm, $1\le q<\infty$.

\begin{theorem}
\label{Th33Besov}
Let $n\in\N$ and $f\in B_{\infty1}^1(\R)\cap B_{\infty1}^n(\R)$. Then $f$ is $n$ times continuously Fr\'{e}chet
$\Sc^q$-differentiable, $1\le q<\infty$, at every $A=A^*$ and
\begin{align}
\label{**2}
D_q^k f(A)(X_1,\ldots,X_k)
=\sum_{\sigma\in \sym_k}T_{f^{[k]}}^{A,\ldots,A}(X_{\sigma(1)},\ldots,X_{\sigma(k)}),
\end{align}
for every $k=1,\ldots,n$, and all $X_1,\ldots,X_k\in \Sc^q$.
\end{theorem}

The proof of Theorem \ref{Th33Besov} is completely analogous to the proof of Theorem \ref{C0} and is based on the counterparts of Lemmas \ref{L1}, \ref{LMOId2}, \ref{L2} for Besov functions stated below.

The following perturbation formula is a higher order extension of the analogous formulas proved in \cite{BS3} and \cite[Theorem 5.1 and Lemma 5.4]{Peller2006} for $n=1$ and $n=2$, respectively.

\begin{lemma}
\label{L37Besov}
Let $n\in\N$, $n\ge 2$, and $f\in B_{\infty1}^n(\R)\cap B_{\infty1}^{n-1}(\R)$.
Let $A_1,\ldots,A_{n-1},A,B$ be self-adjoint operators with $B-A\in \Sc^{q}$, $1\le q<\infty$, let $X_1,\ldots,X_{n-1}\in
\Sc^{q}$. Then, for every $i=1,\ldots,n$,
\begin{align}
\label{37B}
\nonumber
&T_{f^{[n-1]}}^{A_1,\ldots,A_{i-1},A,A_i,\ldots,A_{n-1}}(X_1,\ldots,X_{n-1})
-T_{f^{[n-1]}}^{A_1,\ldots,A_{i-1},B,A_i,\ldots,A_{n-1}}(X_1,\ldots,X_{n-1})\\
&=T_{f^{[n]}}^{A_1,\ldots,A_{i-1},A,B,A_i,\ldots,A_{n-1}}
(X_1,\ldots,X_{i-1},A-B,X_i,\ldots,X_{n-1}).
\end{align}
\end{lemma}

\begin{remark}
We note that the result of Lemma \ref{L37Besov} for $1<q<\infty$ is a particular case of the result of Lemma \ref{L1}. The proof in the case $q=1$ is based on the integral projective tensor product representation of $f^{[k]}$ for $f\in B_{\infty1}^k(\R)$. The latter representation is not available for a general $f\in C^n(\R)$ with $f^{(n-1)},f^{(n)}\in C_b(\R)$, which is treated in Lemma \ref{L1}.
\end{remark}

As an immediate consequence of Lemma \ref{L37Besov}, we obtain the following analog of Lemma \ref{LMOId2}.

\begin{lemma}
\label{L38Besov}
Let $n\in\N$, $n\ge 2$, and $f\in B_{\infty1}^n(\R)\cap B_{\infty1}^{n-1}(\R)$.
Let $A_1,\ldots,A_{n-1},A,B$ be self-adjoint operators with $B-A\in\Sc^{q}$, $1\le q<\infty$,
let $X_1,\ldots,X_{n-1}\in\Sc^{q}$. Then,
\begin{align*}
&T_{f^{[n-1]}}^{A,\ldots,A}(X_1,\ldots,X_{n-1})-T_{f^{[n-1]}}^{B,\ldots,B}(X_1,\ldots,X_{n-1})\\
&=\sum_{i=1}^n T_{f^{[n]}}^{{\bf B}_{i-1},A,B,{\bf A}_{n-i}}(X_1,\ldots,X_{i-1},A-B,X_i,\ldots,X_{n-1}).
\end{align*}
\end{lemma}

We also need the analog of Lemma \ref{L2} proved below.

\begin{lemma}
\label{L310Besov}
Let $n\in\N$, $1\le i\le n$, and $f\in B_{\infty1}^n(\R)$.
Then, for every $\epsilon>0$ there exists $\delta>0$ such that for
all self-adjoint $A,A_1,\ldots,A_n$ and all
$X_1,\ldots,X_n\in\Sc^q$, $1\le q<\infty$, the estimate
\begin{align}
\label{310B}
\nonumber
&\big\|T_{f^{[n]}}^{A_1,\ldots,A_{i-1},A+X,A_i,\ldots,A_n}(X_1,\ldots,X_n)
-T_{f^{[n]}}^{A_1,\ldots,A_{i-1},A,A_i,\ldots,A_n}(X_1,\ldots,X_n)\big\|_q\\
&\le\epsilon\,\|X_1\|_q\cdots\|X_n\|_q
\end{align}
holds whenever $X\in\Sc^q_{sa}$ satisfies $\|X\|_q<\delta$.
\end{lemma}

\begin{proof}

In this proof we adopt the notations \eqref{notation1} and \eqref{notation2}.

For each $m\in\N$, let $f_m=\sum_{k=-m}^m f\ast w_k$, where $w_k$ is defined in \eqref{wk}. Then,
$f_m\in B_{\infty1}^n(\R)\cap B_{\infty1}^{n+1}(\R)$ and
\begin{align*}
\lim_{m\rightarrow\infty}\|f-f_m\|_{B_{\infty1}^n}=0.
\end{align*}
By Lemma \ref{L37Besov},
\begin{align}
\label{L1corB}
\T_{f_m^{[n]}}^{{\bf A},i,A+X}-\T_{f_m^{[n]}}^{{\bf A},i,A}
=\T_{f_m^{[n+1]}}^{{\bf A},i,A+X,A}(X).
\end{align}
By the representation (\ref{L1corB}) and Theorem \ref{Peller},
\begin{align}
\label{310B1}
\big\|\T_{f_m^{[n]}}^{{\bf A},i,A+X}-\T_{f_m^{[n]}}^{{\bf A},i,A}\big\|_q
\le {\rm const}_{n+1}\,\|f_m\|_{B_{\infty1}^{n+1}}\|X\|_q\|X_1\|_q\cdots\|X_n\|_q.
\end{align}

Given $\epsilon>0$, let $m\in\N$ be such that
\begin{align}
\label{310B2}
\|f-f_m\|_{B_{\infty1}^n}<\frac{\epsilon}{4\,{\rm const}_n}.
\end{align}
We have
\begin{align}
\label{310B3}
\big\|\T_{f^{[n]}}^{{\bf A},i,A+X}-\T_{f^{[n]}}^{{\bf A},i,A}\big\|_q\le\big\|T_{(f-f_m)^{[n]}}^{{\bf A},i,A+X}\big\|_q
+\big\|\T_{(f-f_m)^{[n]}}^{{\bf A},i,A}\big\|_q+\big\|T_{f_m^{[n]}}^{{\bf A},i,A+X}-T_{f_m^{[n]}}^{{\bf A},i,A}\big\|_q.
\end{align}
By Theorem \ref{Peller} and \eqref{310B2},
\begin{align}
\label{310B4}
\nonumber
\big\|\T_{(f-f_m)^{[n]}}^{{\bf A},i,A+X}\big\|_q
+\big\|\T_{(f-f_m)^{[n]}}^{{\bf A},i,A}\big\|_q
&\le 2\, {\rm const}_n\,\|f-f_m\|_{B_{\infty1}^n}\|X_1\|_q\cdots\|X_n\|_q\\
&\le\frac12\,\epsilon\,\|X_1\|_q\cdots\|X_n\|_q.
\end{align}
Combining \eqref{310B1}, \eqref{310B3}, \eqref{310B4} guarantees that if
\begin{align*}
\|X\|_q<\delta=\frac{\epsilon}{2\,{\rm const}_{n+1}\,\|f_m\|_{B_{\infty1}^{n+1}}}\,,
\end{align*}
then \eqref{310B} holds.
\end{proof}

\begin{proof}[Proof of Theorem \ref{Th33Besov}]
The result is proved by induction on $n$. Note that by a known property of Besov spaces, $B_{\infty1}^1(\R)\cap B_{\infty1}^n(\R)$ contains $B_{\infty1}^k(\R)$ for every $k=1,\ldots,n$.

The base of induction, case $n=1$, is known.
The Fr\'{e}chet differentiability and (\ref{**2}) for $k=1$ follow
from the G\^{a}teaux differentiability established in \cite[Theorem 2]{Peller1990}, properties of double operator integrals derived in \cite{BS3}, properties of Besov spaces, and the fact that every G\^{a}teaux differentiable function whose derivative is a bounded linear operator continuously depending on the point of differentiation (see \cite[Lemma 1.15]{Schwartz}) is Fr\'{e}chet differentiable.

We have
\begin{align*}
&D_1^1 f(A+X)(X_1)-D_1^1 f(A)(X_1)\\
&=\big(T_{f^{[1]}}^{A+X,A+X}(X_1)-T_{f^{[1]}}^{A+X,A}(X_1)\big)+\big(T_{f^{[1]}}^{A+X,A}(X_1)-T_{f^{[1]}}^{A,A}(X_1)\big).
\end{align*}
Thus, by Lemma \ref{L310Besov}, the Fr\'{e}chet differential of $f$ is continuous.

The inductive step is proved along the lines of the one in Theorem \ref{C0} with replacement of Lemmas \ref{L1}, \ref{LMOId2}, \ref{L2} by Lemmas \ref{L37Besov}, \ref{L38Besov}, \ref{L310Besov}, respectively.
\end{proof}

\bigskip

\paragraph{\bf Acknowledgement.}
A.S. is grateful to the Universit\'{e} de Franche-Comt\'{e}, Besan\c{c}on for hospitality
during her work on this project. The authors are grateful to the anonymous referee for the careful reading and relevant suggestions regarding the presentation of some of the proofs.

\end{document}